\documentclass{elsarticle}
\usepackage{amsmath,amsthm}
\usepackage{bbm}
\newtheorem{df}{Definition}[section]
\newtheorem{prop}[df]{Proposition}
\newtheorem{lem}[df]{Lemma}
\newtheorem{cor}[df]{Corollary}
\newtheorem{rem}[df]{Remark}

\def\id{\operatorname{id}}
\def\im{\operatorname{im}}
\def\span{\operatorname{span}}

\journal{arXiv.org}
\begin{document}
\begin{frontmatter}
\title{Equivariant Algebraic Morse Theory}
\author[ruelle]{Ralf Donau}
\address[ruelle]{Fachbereich Mathematik, Universit\"at Bremen, Bibliothekstra\ss e 1, 28359 Bremen, Germany}
\ead{ruelle@math.uni-bremen.de}
\begin{abstract}
In this paper we develop Algebraic Morse Theory for the case where a group acts on a free chain complex. Algebraic Morse Theory is an adaption of Discrete Morse Theory to free chain complexes.
\end{abstract}
\end{frontmatter}
\section{Introduction}
There exists an equivariant version of the Main Theorem of Discrete Morse Theory, see \cite{freij}. In this paper I present an equivariant version of Theorem 11.24 in \cite[Chapter 11.3]{buch}. We use the same notion of an equivariant acyclic matching as in Equivariant Discrete Morse Theory. An example for working with equivariant acyclic matchings can be found in \cite{donau3}.
\section{Equivariant acyclic matchings}
The definition of an acyclic matching on a poset can be found in \cite{clmap,buch}.
\begin{df}
Let $P$ be a poset and let $G$ be a group acting on $P$. Let $M$ be an acyclic matching on $P$. We call $M$ an \emph{$G$-equivariant acyclic matching} if $(a,b)\in M$ implies $(ga,gb)\in M$ for all $g\in G$ and $a,b\in P$.
\end{df}
There exists a characterization of acyclic matchings by means of order-preserving maps with small fibers, see Definition 11.3 and Theorem 11.4 in \cite[Chapter 11]{buch}. In a similar way we can also characterize $G$-equivariant acyclic matchings by means of order-preserving $G$-maps with small fibers.

For an order-preserving map with small fibers $\varphi$ let $M(\varphi)$ denote its associated acyclic matching which consists of all fibers of cardinality $2$, see \cite[Chapter 11]{buch}.
\begin{prop}
\label{smallfibers}
Let $G$ be a group acting on a finite poset $P$. For any order-preserving $G$-map $\varphi:P\longrightarrow Q$ with small fibers, the acyclic matching $M(\varphi)$ is $G$-equivariant. On the other hand, any $G$-equivariant acyclic matching $M$ on $P$ can be represented as $M=M(\varphi)$, where $\varphi:P\longrightarrow Q$ is an order-preserving $G$-map with small fibers.
\end{prop}
The proof of Proposition \ref{smallfibers} can be found in \cite{donau3}.
\section{The main result}
We consider chain complexes of modules over some fixed commutative ring $R$ with unit. Furthermore we consider group actions on such chain complexes. Let $G$ be a group and let $C_*=(\dots\overset{\partial_{n+2}}\longrightarrow C_{n+1}\overset{\partial_{n+1}}\longrightarrow C_n\overset{\partial_n}\longrightarrow\dots)$ be a finitely generated free chain complex with an action of the group $G$ and let $\Omega=(\Omega_n)_n$ be a $G$-basis of $C_*$, i.e. each $\Omega_n$ is closed under the action of $G$. For $b\in\Omega_n$ let $k_b:C_n\longrightarrow R$, $\sum_{b'\in\Omega_n}\lambda_{b'}b'\longmapsto\lambda_b$ denote the linear function which maps any $x\in C_n$ to the coefficient of $b$ inside the linear representation of $x$.

Let $P(C_*,\Omega):=\bigcup_n\Omega_n$. We define an order relation on $P(C_*,\Omega)$ as follows. For $a\in\Omega_n$ and $b\in\Omega_{n+1}$ we denote the \emph{weight of the covering relation} by $w(b\succ a):=k_a(\partial_n b)$, we set $a\leq b$ if $w(b\succ a)\not=0$. Furthermore for $a\in\Omega_n$ and $b\in\Omega_m$ with $n\leq m$, we set $a\leq b$ if there exists a sequence $(c_i)_{n\leq i\leq m}$ with $c_i\leq c_{i+1}$ for $n\leq i<m$ such that $a=c_n$ and $b=c_m$. This defines a partial order relation on $P(C_*,\Omega)$, which can be easily verified. Notice that $P(C_*,\Omega)$ is a $G$-poset since $\partial_n$ is a $G$-map.

Let $M$ be an $G$-equivariant acyclic matching on $P(C_*,\Omega)$ such that $w(b\succ a)$ is invertible for any $(a,b)\in M$. Let $\varphi$ denote the order-preserving $G$-map with small fibers with $M(\varphi)=M$, which exists by Proposition \ref{smallfibers}.
\begin{df}
For $b\in C_n$ we define the $G$-subcomplex ${\cal A}(Gb)$ as follows.
\[\dots\longrightarrow 0\longrightarrow\span(Gb)\overset{\partial^{Gb}_n}\longrightarrow\span(G\partial_n b)\longrightarrow 0\longrightarrow\dots\]
$\partial^{Gb}_n$ denotes the restriction of $\partial_n:C_n\longrightarrow C_{n-1}$ to $\span(Gb)$, i.e. $\partial^{Gb}_n(x):=\partial_n(x)$ for $x\in\span(Gb)$.
\end{df}
Notice that $\partial^{Gb}_n$ is surjective by construction.
\begin{rem}
\label{poset_orbit}
Let $G$ be a finite group acting on a finite poset $Q$. Let $q\in Q$ and $g\in G$. Then $gq\leq q$ implies $gq=q$. In other words, the elements inside an orbit are not comparable to each other.
\end{rem}
Let $(a,b)\in M$ be a matching pair. Then for any $g\in G$, $a\prec gb$ implies $\varphi(b)=\varphi(a)\leq\varphi(gb)=g\varphi(b)$ which implies $\varphi(b)=\varphi(gb)$ by Remark \ref{poset_orbit}. Hence $b=gb$ since $\varphi$ has small fibers. In particular $k_a(\partial gb)=0$ for $gb\not=b$.
\begin{lem}
\label{iso}
Let $(a,b)\in M$ be a matching pair, assume $b\in C_n$. Then $\partial^{Gb}_n$ is an isomorphism and ${\cal A}(Gb)$ is $G$-homotopy equivalent to the zero complex.
\end{lem}
\begin{proof}
We have to show that $\partial^{Gb}_n$ is injective. Assume $\partial_n(\sum_{b'\in Gb}\lambda_{b'}b')=0$. Then $0=k_a(\sum_{b'\in Gb}\lambda_{b'}\partial_nb')=\lambda_{b}k_a(\partial_nb)$, which implies $\lambda_b=0$, since $k_a(\partial_nb)$ is invertible. This implies $\lambda_{b'}=0$ for all $b'\in Gb$, since $\partial_n$ is a $G$-map.

The composition ${\cal A}(Gb)\longrightarrow0\longrightarrow{\cal A}(Gb)$ is homotop to $\id:{\cal A}(Gb)\longrightarrow{\cal A}(Gb)$ via the $G$-chain homotopy $P=(P_i)$, where
\begin{itemize}
\item $P_{n-1}:=(\partial^{Gb}_n)^{-1}$
\item $P_i:\equiv 0$ for $i\not=n-1$
\end{itemize}
On the other hand $0\longrightarrow{\cal A}(Gb)\longrightarrow 0$ equals $\id:0\longrightarrow 0$. Hence ${\cal A}(Gb)$ is $G$-homotopy equivalent to the zero complex.
\end{proof}
\begin{prop}
\label{eqthm}
Let $M$ be an $G$-equivariant acyclic matching on $P(C_*,\Omega)$ such that $w(b\succ a)$ is invertible. Then there exists a $G$-chain complex $C_*^M$ such that $C_*\cong_G C_*^M\oplus T_*$ where $T_*=\bigoplus_{G(a,b)\in M/G}{\cal A}(Gb)$.
\end{prop}
This in an $G$-equivariant version of Theorem 11.24 in \cite[Chapter 11.3]{buch}. Notice that $C_*^M$ is unique up to $G$-isomorphism.
\begin{proof}
By induction on the number $m=|M/G|$ of orbits in $M$. For $m=0$, we set $C_*^M:=C_*$. Now assume $m>0$. By Proposition \ref{smallfibers}, there exists a finite poset $Q$ and an order-preserving $G$-map with small fibers $\varphi:P(C_*,\Omega)\longrightarrow Q$ such that $M(\varphi)=M$. We consider the following subposet of $Q$:
\[
Q':=\{q\in Q\mid\varphi^{-1}(q)=\{a,b\}\text{ with }a\not=b\}
\]
Let $q\in Q'$ be a minimal element. Assume $a\in\Omega_{n-1}$, $b\in\Omega_n$ and $\varphi^{-1}(q)=\{a,b\}$. We construct a new basis $\widetilde\Omega$ by replacing $\Omega_{n-1}$ and $\Omega_n$ as follows. We construct two $G$-automorphisms $f_{n-1}:C_{n-1}\longrightarrow C_{n-1}$ and $f_n:C_n\longrightarrow C_n$.
\begin{eqnarray*}
f_{n-1}:C_{n-1}\supset\Omega_{n-1}&\longrightarrow&C_{n-1}\\
ga&\longmapsto&g\partial_n b\text{ for all $g\in G$}\\
x&\longmapsto&x\text{ for $x\not\in Ga$}
\end{eqnarray*}
We set $\Omega^{Gb}_{n-1}:=\Omega_{n-1}\setminus Ga$, then $\widetilde\Omega_{n-1}:=f_{n-1}(\Omega_{n-1})=\Omega^{Gb}_{n-1}\cup Gf_{n-1}(a)$ generates $C_{n-1}$ since the linear representation of $f_{n-1}(a)$ is
\[
f_{n-1}(a)=\partial_n b=w(b\succ a)\cdot a+\sum_{x\in\Omega^{Gb}_{n-1}}w(b\succ x)\cdot x
\]
which implies
\[
a=\frac 1{w(b\succ a)}\cdot f_{n-1}(a)-\sum_{x\in\Omega^{Gb}_{n-1}}\frac{w(b\succ x)}{w(b\succ a)}\cdot x
\]
hence $a\in\span(f_{n-1}(\Omega_{n-1}))$. This implies $ga\in\span(f_{n-1}(\Omega_{n-1}))$ for all $g\in G$ since $f_{n-1}$ is a $G$-map by construction.

Assume
\[
\sum_{x\in\Omega^{Gb}_{n-1}}\lambda_xx+\sum_{a'\in Ga}\lambda_{a'}f_{n-1}(a')=0
\]
Then in particular
\begin{eqnarray*}
&k_a\left(\sum_{x\in\Omega^{Gb}_{n-1}}\lambda_xx+\sum_{a'\in Ga}\lambda_{a'}f_{n-1}(a')\right)&=0\\
\Longrightarrow&\lambda_a\cdot k_a(\partial_n b)&=0\\
\Longrightarrow&\lambda_a\cdot w(b\succ a)&=0\\
\Longrightarrow&\lambda_a&=0
\end{eqnarray*}
For any $g\in G$ this implies $\lambda_{ga}=0$, because $f_{n-1}$ is a $G$-map. Since $\Omega^{Gb}_{n-1}$ is linear independent, we have $\lambda_x=0$ for all $x\in\Omega^{Gb}_{n-1}$.
\begin{eqnarray*}
f_n:C_n\supset\Omega_n&\longrightarrow&C_n\\
gb&\longmapsto&gb\text{ for all $g\in G$}\\
x&\longmapsto&x-\sum_{b'\in Gb}w(x\succ\partial_n b')\cdot b'\text{ for $x\not\in Gb$}
\end{eqnarray*}
We set $\Omega_n^{Gb}:=f_n(\Omega_n\setminus Gb)$, then $\widetilde\Omega_n:=f_n(\Omega_n)=\Omega_n^{Gb}\cup Gb$ generates $C_n$, since for all $x\in\Omega_n\setminus Gb$ we have
\[
x=f_n(x)+\sum_{b'\in Gb}w(x\succ\partial b')\cdot b'\in\span(\widetilde\Omega_n)
\]
where $f_n(x)\in\span(\Omega_n^{Gb})$ and $\sum_{b'\in Gb}w(x\succ\partial b')\cdot b'\in\span(Gb)$. Clearly we have $k_x(f_n(x))=1$ and $k_x(f_n(x'))=0$ for all $x'\in\Omega_n\setminus\{x\}$. Assume
\[
\sum_{x'\in\Omega_n\setminus Gb}\lambda_{x'} f_n(x')+\sum_{b'\in Gb}\lambda_{b'}b'=0
\]
Then for all $x\in\Omega_n\setminus Gb$:
\begin{eqnarray*}
&k_x\left(\sum_{x'\in\Omega_n\setminus Gb}\lambda_{x'} f_n(x')+\sum_{b'\in Gb}\lambda_{b'}b'\right)&=0\\
\Longrightarrow&\lambda_xk_x(f_n(x))&=0\\
\Longrightarrow&\lambda_x&=0
\end{eqnarray*}
Since $Gb$ is linear independent, this also implies $\lambda_{gb}=0$ for all $g\in G$.

Hence $\widetilde\Omega_{n-1}$ and $\widetilde\Omega_n$ are bases of $C_{n-1}$ and $C_n$ respectively. Let $\widetilde\Omega$ denote the basis of $C_*$ which is obtained from $\Omega$ by replacing $\Omega_{n-1}$ and $\Omega_n$ with these newly constructed bases. Then $\widetilde\Omega$ has the following properties:
\begin{itemize}
 \item[(1)] $\partial_n(\span(\Omega_n^{Gb}))\subset\span(\Omega_{n-1}^{Gb})$
 \item[(2)] $\im\partial_{n+1}\subset\span(\Omega_n^{Gb})$
\end{itemize}
For $b\in\widetilde\Omega$ let $\widetilde k_b$ denote the linear function which maps an element to the coefficient of $b$ inside its linear representation.

For all $x\in\Omega_n\setminus Gb$ we have
\begin{eqnarray*}
\widetilde k_{f_{n-1}(a)}(\partial_n f_n(x))&=&\widetilde k_{f_{n-1}(a)}(\partial_n x)-\sum_{b'\in Gb}w(x\succ\partial_n b')\widetilde k_{f_{n-1}(a)}(\partial_n b')\\
&=&\widetilde k_{f_{n-1}(a)}(\partial_n x)-w(x\succ\partial_n b)\\
&=&\widetilde k_{f_{n-1}(a)}(\partial_n x)-w(x\succ f_{n-1}(a))\\
&=&0
\end{eqnarray*}
Since $f_n$ is a $G$-map, we also have $\widetilde k_{f_{n-1}(ga)}(\partial_n f_n(x))=0$ for all $g\in G$ and all $x\in\Omega_n\setminus Gb$. Hence (1) is proven.

We have $\widetilde k_{f_{n-1}(a)}(\partial_n f_n(b'))=\widetilde k_{f_n(b)}(f_n(b'))$ for all $b'\in Gb$ by construction of $f_{n-1}$, on the other hand we have $0=\widetilde k_{f_{n-1}(a)}(\partial_n x)=\widetilde k_{f_n(b)}(x)$ for all $x\in\Omega_n^{Gb}$ as shown before. Hence $\widetilde k_{f_{n-1}(a)}(\partial_n z)=\widetilde k_{f_n(b)}(z)$ for all $z\in C_n$. In particular $0=\widetilde k_{f_{n-1}(a)}(\partial_n\partial_{n+1}(y))=\widetilde k_{f_n(b)}(\partial_{n+1}(y))$ for all $y\in C_{n+1}$. Since $C_*$ is a $G$-complex, we also have $\widetilde k_{f_n(gb)}(\partial_{n+1}(y))=0$ for all $y\in C_{n+1}$ and all $g\in G$. Hence (2) is proven.

We obtain the following two $G$-subcomplexes ${\cal A}(Gb)$ and $C_*^{Gb}$:
\[\dots\longrightarrow 0\longrightarrow\span(Gb)\overset{\partial_n^{Gb}}\longrightarrow\span(Gf_{n-1}(a))\longrightarrow 0\longrightarrow\dots\]
\[\dots\longrightarrow C_{n+1}\overset{\partial_{n+1}}\longrightarrow\span(\Omega_n^{Gb})\overset{\partial_n}\longrightarrow\span(\Omega_{n-1}^{Gb})\overset{\partial_{n-1}}\longrightarrow C_{n-2}\longrightarrow\dots\]

Clearly we have $C_*\cong_G C_*^{Gb}\oplus{\cal A}(Gb)$. Furthermore $C_*^{Gb}$ is a free $G$-chain complex with $G$-basis $\Omega^{Gb}$, which is obtained from $\Omega$ by replacing $\Omega_n$ and $\Omega_{n-1}$ by $\Omega_n^{Gb}$ and $\Omega_{n-1}^{Gb}$ respectively.

Now we construct an $G$-equivariant acyclic matching $M^{Gb}$ on $P(C_*^{Gb},\Omega^{Gb})$ such that $w(b\succ a)$ is invertible for any $(a,b)\in M^{Gb}$. Let $R\subset P(C_*,\Omega)$ be the set of all elements appearing in any matching pair in $M\setminus G(a,b)$. We define the map
\begin{eqnarray*}
f:R&\longrightarrow&P(C_*^{Gb},\Omega^{Gb})\\
x&\longmapsto&f_n(x)\text{ for $x\in\Omega_n\setminus Gb$}\\
x&\longmapsto&x\text{ else}
\end{eqnarray*}
which has the following property: Let $x,y\in R$ such that $x\succ y$, then we have $w(f(x)\succ f(y))=w(x\succ y)$. This is clear for $x,y\not\in C_n$. It remains to prove the two other cases:

Case $x\in\Omega_n$: $\varphi(b)=q$ is minimal in $Q'$ by choice. We have $b\not\succ y$, since $b>y$ implies $\varphi(b)>\varphi(y)\in Q'$ which contradicts the minimality of $q$. Because $C_*^{Gb}$ is a $G$-complex, this implies $w(gb\succ y)=0$ for all $g\in G$. Hence $w(f(x)\succ f(y))=w(f(x)\succ y)=w(x\succ y)-\sum_{b'\in Gb} w(x\succ\partial_n b')\cdot w(b'\succ y)=w(x\succ y)$.

Case $y\in\Omega_n$: As mentioned (how $k_x$ behaves on $\widetilde\Omega$) in the proof of the injectivity of $f_n$, we have $\widetilde k_{f(y)}(\partial_{n+1} x)=k_y(\partial_{n+1} x)$. Hence $w(f(x)\succ f(y))=w(x\succ f(y))=w(x\succ y)$.

We set
\[
M^{Gb}:=\{(f(a'),f(b')\mid(a',b')\in M\setminus G(a,b)\}
\]
By the property of $f$, any cycle in $M^{Gb}$ would map to a cycle in $M$ under $f^{-1}$. It is easy to see that $M^{Gb}$ has all desired properties. By induction hypothesis there exists a $G$-chain complex $X_*$ such that
\begin{eqnarray*}
C_*^{Gb}&\cong_G&X_*\oplus\bigoplus_{G(a',b')\in M^{Gb}/G}{\cal A}(Gb')\\
&=&X_*\oplus\bigoplus_{G(a',b')\in M\setminus G(a,b)/G}{\cal A}(Gf_n(b'))
\end{eqnarray*}
$f_n$ induces a $G$-chain isomorphism $(\varphi_n)_n:{\cal A}(Gb')\longrightarrow{\cal A}(Gf_n(b'))$ for any $b'\in R$, where
\begin{eqnarray*}
\varphi_n:\span(Gb')&\longrightarrow&\span(Gf_n(b'))\\
x&\longmapsto&f_n(x)
\end{eqnarray*}
\begin{eqnarray*}
\varphi_{n-1}:\span(G\partial_nb')&\longrightarrow&\span(G\partial_nf_n(b'))\\
x&\longmapsto&(\partial_n\circ f_n)((\partial^{Gb'}_n)^{-1}(x))
\end{eqnarray*}
By Lemma \ref{iso}, the restrictions of $\partial_n$ used in the definition of $\varphi_{n-1}$ are isomorphisms.

We set $C_*^M:=X_*$.
\begin{eqnarray*}
C_*\cong_G C_*^{Gb}\oplus{\cal A}(Gb)&\cong_G&\left(C_*^M\oplus\bigoplus_{G(a',b')\in M\setminus G(a,b)/G}{\cal A}(Gb')\right)\oplus{\cal A}(Gb)\\
&=&C_*^M\oplus\bigoplus_{G(a',b')\in M/G}{\cal A}(Gb')
\end{eqnarray*}
\end{proof}
In the case $G$ is the trivial group, our result is the same as Theorem 11.24 in \cite[Chapter 11.3]{buch}, where $C_*^M$ is the Morse complex defined in \cite[Chapter 11.3]{buch}. By forgetting the action of $G$ we obtain
\[
\bigoplus_{G(a,b)\in M/G}{\cal A}(Gb)\cong\bigoplus_{(a,b)\in M}{\cal A}(b)
\]
Since the Morse complex for the matching $M$ is also a $G$-complex, it is usable in the $G$-equivariant case too.
\begin{cor}
$C_*$ is $G$-homotopy equivalent to $C_*^M$.
\end{cor}
\begin{proof}
$T_*$ is $G$-homotopy equivalent to the zero complex by Lemma \ref{iso}. Hence $C_*^M\oplus 0$ is $G$-homotopy equivalent to $C_*^M\oplus T_*$ which is $G$-isomorphic to $C_*$ by Proposition \ref{eqthm}, in particular $G$-homotopy equivalent.
\end{proof}
\bibliographystyle{model1-num-names}

\end{document}